\documentclass[11pt,a4paper,reqno]{amsart}

\usepackage{amsmath}
\usepackage{amsfonts}
\usepackage{a4wide}
\usepackage{amssymb}
\usepackage{amsthm}
\usepackage{mathrsfs}
\usepackage[colorlinks, citecolor=blue, linkcolor=red]{hyperref}

\theoremstyle{plain}

\newtheorem{teo}{Theorem}[section]
\newtheorem{lemma}[teo]{Lemma}
\newtheorem{prop}[teo]{Proposition}
\newtheorem{cor}[teo]{Corollary}
\newtheorem{ackn}{Acknowledgments\!}

\theoremstyle{definition}

\theoremstyle{remark}

\numberwithin{equation}{section}

\def\Ft{\mathcal{F}_{t}}

\def\M1{\mathscr{M}_{1}}

\def\eps{\varepsilon}

\newcommand{\kn}{\mathbin{\bigcirc\mkern-15mu\wedge}}

\title[A variational characterization of flat spaces in dimension three]{A variational characterization of flat spaces \\in dimension three}

\author[Giovanni Catino]{Giovanni Catino}
\address[Giovanni Catino]{Dipartimento di Matematica, Politecnico di Milano, Piazza Leonardo da Vinci 32, 20133 Milano, Italy}
\email[]{giovanni.catino@polimi.it}

\author[Paolo Mastrolia]{Paolo Mastrolia}
\address[Paolo Mastrolia]{Dipartimento di Matematica, Universit\`{a} degli Studi di Milano, Italy, Via Cesare Saldini 50, 20133 Milano, Italy}
\email[]{paolo.mastrolia@unimi.it}

\author[Dario D. Monticelli]{Dario D. Monticelli}
\address[Dario D. Monticelli]{Dipartimento di Matematica, Universit\`{a} degli Studi di Milano, Via Cesare Saldini 50, 20133 Milano, Italy}
\email[]{dario.monticelli@unimi.it}

\date{\today}

\begin{document}

\begin{abstract} In this short note we prove that, in dimension three, flat metrics are the only complete metrics with non-negative scalar curvature which are critical for the $\sigma_{2}$-curvature functional.
\end{abstract}

\maketitle

\begin{center}


\medskip

\centerline{\bf AMS subject classification:  53C24, 53C25}

\end{center}

\

\section{Introduction}

Let $(M^{n},g)$ be a Riemannian manifold of dimension $n\geq 3$. To fix the notation, we recall the decomposition of the Riemann curvature tensor of a metric $g$ into the Weyl, Ricci and scalar curvature component
$$
Rm \,= \, W + \frac{1}{n-2} Ric \kn g - \frac{1}{(n-1)(n-2)}R \, g \kn g \,,
$$
where $\kn$ denotes the Kulkarni-Nomizu product. It is well known~\cite{hilbert} that Einstein metrics are critical points for the Einstein-Hilbert functional
$$
\mathcal{H} \,=\, \int R \,dV
$$
on the space of unit volume metrics $\M1(M^{n})$. From this perspective, it is natural to study  canonical metrics which arise as solutions of the Euler-Lagrange equations for more general curvature functionals.  In~\cite{berger3}, Berger commenced the study of Riemannian functionals which are quadratic in the curvature (see~\cite[Chapter 4]{besse} and~\cite{smot} for surveys). A basis for the space of quadratic curvature functionals is given by
$$
\mathcal{W} \,=\, \int |W|^{2} dV\, \quad \mathcal{\rho} \,=\, \int |Ric|^{2} dV\, \quad\, \mathcal{S} \,=\, \int R^{2} dV \,.
$$
All such functionals, which also naturally arise as total actions in certain gravitational field theories in physics, have been deeply studied in the last years by many authors, in particular on compact Riemannian manifolds with normalized volume (for instance, see~\cite{berger3, besse, lamontagne2, lamontagne1, ander2, gurvia3, gurvia1, gurvia2, cat3} and references therein).

On the other hand, the study of critical metrics for quadratic curvature functionals also has a lot of interest in the non-compact setting. For instance, Anderson in~\cite{ander3} proved that every complete three-dimensional critical metric for the Ricci functional $\mathcal{\rho}$ with non-negative scalar curvature is flat, whereas in~\cite{cat2} the first author showed a characterization of complete critical metrics for $\mathcal{S}$ with non-negative scalar curvature in every dimension.

In this paper we focus our attention on the three dimensional case and consider the $\sigma_{2}$-curvature functional
$$
\mathscr{F}_{2} \,=\, \int \sigma_{2}(A) \, dV \,,
$$
where $\sigma_{2}(A)$ denotes the second elementary symmetric function of the eigenvalues of the Schouten tensor $A = Ric - \frac{1}{4}R\,g$. This functional was first considered by Gursky and Viaclovsky in the compact three dimensional case. In~\cite{gurvia3} they proved a beautiful characterization theorem of space forms as critical metrics for $\mathscr{F}_{2}$ on $\M1(M^{3})$ with non-negative energy $\mathscr{F}_{2}\geq 0$. 

The main result of this paper is the following variational characterization of three dimensional flat spaces.

\begin{teo}\label{t-main}
Let $(M^{3},g)$ be a complete critical metric for $\mathscr{F}_{2}$ with non-negative scalar curvature. Then $(M^{3},g)$ is flat.
\end{teo}

We remark the fact that the non-negativity condition on the scalar curvature cannot be removed. This is clear from the example in~\cite{gurvia3} where the authors exhibit an explicit 
family of critical metrics for $\mathscr{F}_{2}$ on $\mathbb{R}^{3}$. For instance, the metric given in standard coordinates by
$$
g \,=\, dx^{2} + dy^{2} + \big(1+x^{2}+y^{2}\big)^{2} dz^{2}  
$$
is complete, critical and has strictly negative scalar curvature $R=-\frac{8}{1+x^{2}+y^{2}}$.

\

\section{The Euler-Lagrange equation for $\Ft$} \label{s-euler}

In this section we will compute the Euler-Lagrange equation satisfied by critical metrics for $\mathscr{F}_{2}$. To begin, we observe that, in dimension $n\geq 3$, the second elementary symmetric function of the eigenvalues of the Schouten tensor
$$
A \,=\, \frac{1}{n-2} \Big( Ric - \frac{1}{2(n-1)} R\,g\Big)
$$
can be written as
$$
\sigma_{2}(A) \,=\, -\frac{1}{2(n-2)^{2}}|Ric|^{2}+\frac{n}{8(n-1)(n-2)^{2}}R^{2} \,.
$$
In particular, the functional $\mathscr{F}_{2}$ is proportional to a general quadratic functional of the form
$$
\Ft \,=\, \int |Ric|^{2} dV +  t \int R^{2}  dV \,,
$$
with the choice $t=-\frac{n}{4(n-1)}$ (see also~\cite{gurvia1, cat3}). The gradients of the functionals $\mathcal{\rho}$ and $\mathcal{S}$, computed using compactly supported variations, are given by (see~\cite[Proposition 4.66]{besse})
$$
(\nabla \mathcal{\rho})_{ij} \,=\, -\Delta R_{ij} - 2 R_{ikjl}R_{kl}+\nabla^{2}_{ij} R - \frac{1}{2}(\Delta R)g_{ij} + \frac{1}{2} |Ric|^{2} g_{ij} \,,
$$

$$
(\nabla \mathcal{S})_{ij} \,=\, 2 \nabla^{2}_{ij} R - 2(\Delta R) g_{ij} -2 R R_{ij} + \frac{1}{2} R^{2} g_{ij} \,.
$$
Hence, the gradient of $\Ft$ reads 
$$
(\nabla \Ft)_{ij} \,=\, -\Delta R_{ij} +(1+2t)\nabla^{2}_{ij} R - \frac{1+4t}{2}(\Delta R)g_{ij} + \frac{1}{2} \Big( |Ric|^{2}+ t R^{2}\Big) g_{ij} - 2 R_{ikjl}R_{kl} -2t R R_{ij} \,.
$$
Tracing the equation $(\nabla \Ft) = 0$, we obtain
$$
 \Big(n+4(n-1)t\Big)\Delta R \,=\,(n-4)\Big(|Ric|^{2}+ t R^{2}\Big) \,.
$$ 
Defining the tensor $E$ to be the trace-less Ricci tensor, $E_{ij} = R_{ij} - \frac{1}{n} R g_{ij}$, we obtain the Euler-Lagrange equation of critical metrics for $\Ft$.

\begin{prop}\label{p-eulern}
Let $M^{n}$ be a complete manifold of dimension $n\geq 3$. A metric $g$ is critical for $\Ft$ if and only if it satisfies the following equations
\begin{eqnarray*}
\Delta E_{ij} &=&  (1+2t) \nabla^{2}_{ij} R - \frac{n+2+4nt}{2n} (\Delta R) g_{ij} - 2 R_{ikjl} E_{kl} -\frac{2+2nt}{n} R E_{ij} \\&&+  \frac{1}{2}\Big( |Ric|^{2}-\frac{4-n(n-4)t}{n^{2}} R^{2} \Big) g_{ij} \,,
\end{eqnarray*}
\begin{equation*}
\Big(n+4(n-1)t\Big)\Delta R \,=\, (n-4) \Big( |Ric|^{2} + t R^{2} \Big) \,.
\end{equation*}
\end{prop}

In dimension three we recall the decomposition of the Riemann curvature tensor
$$
R_{ikjl} \,=\, E_{ij}g_{kl}-E_{il}g_{jk}+E_{kl}g_{ij}-E_{kj}g_{il} + \frac{1}{6}\, R \,(g_{ij}g_{kl}-g_{il}g_{jk}) \,.
$$
In particular, one has
$$
R_{ikjl} E_{kl} \,=\, -2 E_{ip}E_{jp} -\frac{1}{6}R E_{ij} +  |E|^{2}g_{ij} \,.
$$

Hence, if $n=3$ and $t=-n/4(n-1)=-3/8$, one has
$$
\mathscr{F}_{2} \,=\, -\frac{1}{2} \mathcal{F}_{-3/8} \,,
$$
and the following formulas hold.

\begin{prop}\label{p-euler3}
Let $M^{3}$ be a complete manifold of dimension three. A metric $g$ is critical for $\mathscr{F}_{2}$ if and only if it satisfies the following equations
\begin{eqnarray}\label{eq1}
\Delta E_{ij} \,=\,  \frac{1}{4} \nabla^{2}_{ij} R - \frac{1}{12} (\Delta R) g_{ij} + 4 E_{ip} E_{jp} +\frac{5}{12} R E_{ij} -\frac{1}{2}\Big(3|E|^{2} - \frac{1}{72} R^{2}\Big) g_{ij} \,,
\end{eqnarray}

\begin{equation} \label{eq2}
-2\,\sigma_{2}(A) \,=\,  |Ric|^{2} - \frac{3}{8} R^{2}  \,=\, |E|^{2} - \frac{1}{24}R^{2} \,=\, 0 \,.
\end{equation}
\end{prop}

Now, contracting equation~\eqref{eq1} with $E$ we obtain the following Weitzenb\"ock formula.

\begin{cor}
Let $M^{3}$ be a complete manifold of dimension three. If $g$ is a critical metric for $\mathscr{F}_{2}$, then the following formula holds
\begin{equation}\label{eq-w}
\frac{1}{2}\Delta |E|^{2} \,=\,  |\nabla E|^{2} + \frac{1}{4} E_{ij}\nabla^{2}_{ij} R +4 E_{ip} E_{jp}E_{ij} +\frac{5}{12} R |E|^{2} \,.
\end{equation}
\end{cor}

\

\section{Proof of Theorem~\ref{t-main}}

In this section we will prove Theorem~\ref{t-main}. We  assume that $(M^{3},g)$ is a critical metric for $\mathscr{F}_{2}$ with non-negative scalar curvature $R\geq 0$. In particular, $g$ has zero $\sigma_{2}$-curvature, i.e. $|E|^{2} = \frac{1}{24} R^{2}$ and we obtain
$$
\frac{1}{2} \Delta |E|^{2} \,=\, \frac{1}{48} \Delta R^{2} \,=\, \frac{1}{24} R \Delta R + \frac{1}{24} |\nabla R|^{2} \,.
$$
Putting together this equation with~\eqref{eq-w}, we obtain that the scalar curvature $R$ satisfies the following PDE

\begin{equation}\label{eq-pde}
\frac{1}{24} \Big( R g_{ij} - 6 E_{ij} \Big) \nabla^{2}_{ij} R \,=\, |\nabla E|^{2} - \frac{1}{24}|\nabla R|^{2} + 4 E_{ip} E_{jp}E_{ij} +\frac{5}{12} R |E|^{2} \,.
\end{equation}

To begin, we need the following purely algebraic lemmas.

\begin{lemma} \label{l1}
Let $(M^{3},g)$ be a Riemannian manifold with $R\geq 0$ and $\sigma_{2}(A) \geq 0$. Then, 
$$
R g_{ij} \,\geq\, 6 E_{ij}
$$
and $g$ has non-negative sectional curvature.
\end{lemma}
\begin{proof} Let $\lambda_{1}\leq \lambda_{2} \leq \lambda_{3}$ be the eigenvalues of the Schouten tensor $A=E + \frac{1}{12} R g $ at some point. Then, by assumptions, we have
$$
4 R \,=\, tr(A) \,=\, \lambda_{1} + \lambda_{2} + \lambda_{3} \,\geq\, 0 \quad \quad \hbox{and}\quad \quad  \sigma_{2}(A)\,=\,\lambda_{1} \lambda_{2} + \lambda_{1}\lambda_{3} +\lambda_{2} \lambda_{3} \geq 0 \,.
$$
We want to show that $E \leq \frac{1}{6} R g$ or, equivalently, that 
$$
 A \,\leq\, \frac{1}{4} R g \,=\, \hbox{tr}(A) g \,.
$$ 
Hence, it is sufficient to prove that $\lambda_{3} \leq \hbox{tr}(A) = \lambda_{1} + \lambda_{2} + \lambda_{3}$, i.e. $\lambda_{1} + \lambda_{2} \geq 0$. But this follows by
$$
0 \,\leq\, \lambda_{1} \lambda_{2} + \lambda_{1}\lambda_{3} +\lambda_{2} \lambda_{3} \,=\, ( \lambda_{1} + \lambda_{2} ) \hbox{tr}(A) - (\lambda_{1}^{2} + \lambda_{2}^{2}+\lambda_{1}\lambda_{2} ) \,\leq\, ( \lambda_{1} + \lambda_{2} ) \hbox{tr}(A) \,.
$$
The fact that $g$ has non-negative sectional curvature follows from the decomposition of the Riemann tensor in dimension three and the curvature condition $Ric \,\leq\, \frac{1}{2} R g $ (for instance see~\cite[Corollary 8.2]{hamilton1}).
\end{proof}

\begin{lemma}\label{l2}
Let $(M^{3},g)$ be a Riemannian manifold with $R\geq 0$ and $\sigma_{2}(A) = \hbox{const} \geq 0 $. Then, 
$$
|\nabla E|^{2} \,\geq\, \frac{1}{24}|\nabla R|^{2} \,.
$$
\end{lemma}
\begin{proof} We will follow the proof in~\cite[Lemma 4.1]{gurvia3}. Let $p$ be a point in $M^{3}$. If $R(p)=0$, then $\nabla R=0$ and the lemma follows. So we can assume that $R(p)>0$. Since $-2\,\sigma_{2}(A)= |E|^{2} - \frac{1}{24} R^{2} = const$, one has
\begin{equation}\label{eqp}
|E|^{2} |\nabla |E||^{2} \,=\, \frac{1}{576} R^{2} |\nabla R|^{2} \,. 
\end{equation}
By Kato's inequality $|\nabla |E||^{2} \leq |\nabla E|^{2}$ and the fact that $|E|^{2} \leq \frac{1}{24} R^{2}$, we obtain
$$
|E|^{2} |\nabla E|^{2} \, \geq \, \frac{1}{576} R^{2} |\nabla R|^{2} \,\geq \, \frac{1}{24} |E|^{2} |\nabla R|^{2} \,.
$$
Dividing by $|E|^{2}(p)\neq 0$ the result follows, otherwise, if $|E|(p)=0$, $(\nabla R)(p)=0$ from equation~\eqref{eqp}, and we conclude.
\end{proof}

\begin{lemma}\label{l3}
Let $(M^{3},g)$ be a Riemannian manifold. Then,
$$
E_{ip} E_{jp}E_{ij} \,\geq\, - \frac{1}{\sqrt{6}} |E|^{3} \,.
$$
\end{lemma}
\begin{proof} For a proof of this lemma, for instance, see~\cite[Lemma 4.2]{gurvia3}.
\end{proof}

\begin{cor}\label{c-pdesub}
Let $(M^{3},g)$ be a complete critical metric for $\mathscr{F}_{2}$ with non-negative scalar curvature. Then, $R g_{ij} \geq 6 E_{ij}$, $g$ has non-negative sectional curvature and the scalar curvature satisfies the following differential inequality 
$$
 \Big( R g_{ij} - 6 E_{ij} \Big) \nabla^{2}_{ij} R \,\geq\, \frac{1}{12} R^{3} \,.
$$
\end{cor}
\begin{proof} From equation~\eqref{eq-pde}, combining lemmas~\ref{l1}~\ref{l2} and~\ref{l3}, we obtain
$$
\frac{1}{24}\Big( R g_{ij} - 6 E_{ij} \Big) \nabla^{2}_{ij} R \, \geq\, \frac{5}{12} R |E|^{2} -\frac{4}{\sqrt{6}} |E|^{3} \,=\, |E|^{2} \Big( \frac{5}{12}R - \frac{4}{\sqrt{6}} |E| \Big) \,=\, \frac{1}{288} R^{3} \,,
$$
where in the last equality we have used the fact that $|E|^{2} = \frac{1}{24} R^{2}$.
\end{proof}

Now we can prove Theorem~\ref{t-main}. Clearly, if $M^{3}$ is compact, from Corollary~\ref{c-pdesub}, at a maximum point of $R$ we obtain $R\leq 0$. Hence, $R\equiv 0$ on $M^{3}$ and from equation~\eqref{eq2}, $Ric\equiv 0$ and the metric is flat. So, from now on, we will assume the manifold $M^{3}$ to be non-compact.

Choose now $\phi=\phi(r)$ be a function of the distance $r$ to a fixed point $O\in M^{3}$ and let $B_{s}(O)$ be a geodesic ball of radius $s>0$. We denote by $C_{O}$ the cut locus at the point $O$ and we choose $\phi$ satisfying the following properties: $\phi=1$ on $B_{s}(O)$, $\phi=0$ on $M^{3} \setminus B_{2s}(O)$ and
$$
-\frac{c}{s} \phi^{3/4} \,\leq\, \phi' \,\leq\, 0  \quad \quad \hbox{and} \quad \quad |\phi''| \,\leq\, \frac{c}{s^{2}} \, \phi^{1/2} \,
$$
on $B_{2s}(O) \setminus B_{s}(O)$ for some positive constant $c>0$. In particular, $\phi$ is $C^{3}$ in $M^{3}\setminus C_{O}$. Let $u := R \phi$ and $a_{ij}:= \Big( R g_{ij} - 6 E_{ij} \Big)$. From Corollary~\ref{c-pdesub}, we know that $a_{ij}\geq 0$ and we obtain
\begin{eqnarray} \label{eqn3}
 a_{ij} \nabla^{2}_{ij} u &=& a_{ij} \Big( \phi \nabla^{2}_{ij}R + R \nabla^{2}_{ij} \phi + 2 \nabla_{i}R \nabla_{j} \phi \Big) \\ \nonumber
 &\geq& \frac{1}{12} R^{3} \phi + R \phi' a_{ij} \nabla^{2}_{ij} r + R \phi'' a(\nabla r, \nabla r) + 2 a(\nabla R, \nabla \phi) \,.
\end{eqnarray}
Now, let $p_{0}$ be a maximum point of $u$ and assume that $p_{0}\notin C_{O}$. If $\phi(p_{0})=0$, then $u\equiv 0$ and then $R\equiv 0 $ on $B_{2s}(O)$. Hence, from now on we will assume  $\phi(p_{0})>0$. Then, at $p_{0}$, we have $\nabla u (p_{0})=0$ and $\nabla^{2}_{ij} u (p_{0}) \leq 0$. In particular, at $p_{0}$, one has
$$
\nabla R (p_{0}) \,=\, -\frac{R(p_{0})}{\phi(p_{0})} \nabla \phi (p_{0}) \,.
$$ 
Moreover, since $a_{ij} \geq 0$, for every vector field $X$, one has $a(X,X) \leq \hbox{tr}(a)|X|^{2} = 3 R |X|^{2}$. On the other hand, from standard Hessian comparison theorem, since $g$ has non-negative sectional curvature, we know that on $M^{3}\setminus C_{O}$, one has $ \nabla^{2}_{ij} r \,\leq\, \frac{1}{r} g_{ij} \,.$ Thus, from equation~\eqref{eqn3}, at $p_{0}$, we get
\begin{eqnarray*}
0 &\geq & \frac{1}{12} R^{3} \phi + R \phi' a_{ij} \nabla^{2}_{ij} r + R \phi'' a(\nabla r, \nabla r) - 2 \frac{R}{\phi} a(\nabla \phi, \nabla \phi) \\
&\geq & \frac{1}{12} R^{3} \phi - \Big(\frac{|\phi'|}{r}+|\phi''| + 2 \frac{(\phi')^{2}}{\phi} \Big)R \, \hbox{tr}(a) \\
&\geq & \frac{1}{12} R^{3} \phi - 3\Big(\frac{|\phi'|}{s}+|\phi''| + 2 \frac{(\phi')^{2}}{\phi} \Big) R^{2} \,,
\end{eqnarray*}
where, in the last inequality, we have used the fact that $r\geq s$ on $B_{2s}(O)\setminus B_{s}(O)$, i.e. where $\phi' \neq 0$. From the assumptions on the cut-off function $\phi$, we obtain, at the maximum point $p_{0}$
$$
0 \,\, \geq \,\, \frac{1}{12} R^{2} \phi^{1/2} \Big( R\phi^{1/2} - \frac{c'}{s^{2}} \Big)  
$$ 
for some positive constant $c'>0$. Thus, we have proved that, if $p_{0} \notin C_{O}$, then for every $p\in B_{2s}(O)$
$$
u(p)\,\leq \,u(p_{0}) \,=\, R(p_{0}) \phi (p_{0}) \,\leq\, \frac{c'}{s^{2}}\,. 
$$
If $p_{0}\in C_{O}$ we argue as follows (this trick is usually referred to Calabi). Let $\gamma:[0,L] \rightarrow M^{3}$, where $L=d(p_{0},O)$, be a minimal geodesic joining $O$ to $p_{0}$, the maximum point of $u$. Let $p_{\eps}=\gamma(\eps)$ for some $\eps>0$. Define now
$$
u_{\eps}(x) \,=\, R (x)\phi \big(d(x,p_{\eps})+\eps\big)\,.
$$
Since $d(x,p_{\eps})+\eps\geq d(x,O)$ and $d(p_{0},p_{\eps}) + \eps = d(p_{0},O)$, it is easy to see that $u_{\eps}(p_{0}) = u (p_{0})$ and 
$$
u_{\eps}(x) \,\leq \, u(x) \quad\quad\hbox{for all }\,x \in M^{3} \,, 
$$
since $\phi'\leq 0$. Hence $p_{0}$ is also a maximum point for $u_{\eps}$. Moreover, $p_{0}\notin C_{p_{\eps}}$, so the function $d(x,p_{\eps})$ is smooth in a neighborhood of $p_{0}$ and we can apply the maximum principle argument as before to obtain an estimate for $u_{\eps}(p_{0})$ which depends on $\eps$. Taking the limit as $\eps\rightarrow 0$, we obtain the desired estimate on $u$.

By letting $s\rightarrow +\infty$ we obtain $u\equiv 0$, so $R\equiv 0$. From equation~\eqref{eq2} we have $E\equiv 0$ and so $Ric\equiv 0$ and Theorem~\ref{t-main} follows.

This concludes the proof of Theorem~\ref{t-main}.

\

\

\begin{ackn} The authors are members of the Gruppo Nazionale per
l'Analisi Matematica, la Probabilit\`{a} e le loro Applicazioni (GNAMPA) of the Istituto Nazionale di Alta Matematica (INdAM). The first author is supported by the GNAMPA project ``Equazioni di evoluzione geometriche e strutture di tipo Einstein''. The second and the third authors are supported by the GNAMPA project ``Analisi Globale ed Operatori Degeneri''. The third author is partially supported by FSE, Regione Lombardia. 
\end{ackn}

\

\bibliographystyle{amsplain}
\bibliography{biblio}

\bigskip

\parindent=0pt

\

\end{document}